\pdfoutput=1

\documentclass[a4paper]{article}
\usepackage[T1]{fontenc}
\usepackage{array}
\usepackage{multirow}
\usepackage{amsmath}
\usepackage{amssymb}
\usepackage{stmaryrd}
\usepackage{graphicx}
\usepackage{wasysym}
\usepackage{diagbox}
\usepackage{lipsum}
\usepackage{amsthm}
\usepackage{enumerate}
\usepackage[shortlabels]{enumitem}
\usepackage[colorlinks,
linkcolor=blue,
anchorcolor=blue,
citecolor=blue]{hyperref}
\usepackage[symbol]{footmisc}

\theoremstyle{definition}
\newtheorem{thm}{Theorem}
\newtheorem{defn}{Definition}

\newtheorem{cor}{Corollary}[thm]
\newtheorem{prop}{Proposition}
\newtheorem{rem}{Remark}

\usepackage[english]{babel}

\begin{document}

\title{On Rigid Origami \uppercase\expandafter{\romannumeral1}: Piecewise-planar Paper with Straight-line Creases}

\author{Zeyuan He, Simon D. Guest\footnote{zh299@cam.ac.uk, sdg@eng.cam.ac.uk. Department of Engineering, University of Cambridge, Cambridge CB2 1PZ, United Kingdom.}}

\maketitle

\begin{abstract}
We develop a theoretical framework for rigid origami, and show how this framework can be used to connect rigid origami and results from cognate areas, such as the rigidity theory, graph theory, linkage folding and computer science. First, we give definitions on important concepts in rigid origami, then focus on how to describe the configuration space of a creased paper. The shape and 0-connectedness of the configuration space are analysed using algebraic, geometric and numeric methods, where the key results from each method are gathered and reviewed.
\end{abstract}

\begin{small}
	{$\; \; \,$ \textbf{Keywords}:} rigid-foldability, folding, configuration
\end{small}

	
\section{Introduction}
This article develops a general theoretical framework for rigid origami, and uses this to gather and review the progress that researchers have made on the theory of rigid origami, including other related areas, such as rigidity theory, graph theory, linkage folding, and computer science. 

Origami has been used for many different physical models, as a recent review \cite{callens_flat_2017} shows. Sometimes a "rigid" origami model is required where all the deformation is concentrated on the creases. A rigid origami model is usually considered to be a system of rigid panels that are able to rotate around their common boundaries and has been applied to many areas across different length scales \cite{debnath_origami_2013}. These successful applications have inspired us to focus on the fundamental theory of rigid origami. Ultimately, we are considering two problems: first, the positive problem, which is to find useful sufficient and necessary conditions for a creased paper to be rigid-foldable; second, the inverse problem, which is to approximate a target surface by rigid origami.

The paper is organized as follows. In Sections 2 and 3 we will clarify the definitions of relevant concepts, such as what we mean by paper and rigid-foldability. In Sections 4, 5 and 6 we will show three methods that can



\noindent be applied to study rigid origami. Specifically, an algebraic method (linked to rigidity theory and graph theory), a geometric method (linked to linkage folding), and a numerical method (linked to computer science). Some comments and a brief discussion on some important downstream open problems on rigid origami conclude the paper.

\section{Modelling}


In this section we start with some basic definitions for \textit{origami}. Although the idea of \textit{folding} can be precisely described by isometry excluding Euclidean motion, the definition of paper needs to be carefully considered: we want our mathematical definition to correspond with the commonly understood properties of a paper in the physical world. A paper should not just be a surface in $\mathbb{R}^3$. At any point, there might be contact of different parts of a paper, although crossing is not allowed. We introduce the idea of a \textit{generalized surface} in Definition \ref{defn: generalized} that allows multiple layers local to a point, and in Definition \ref{defn: basic} exclude all the crossing cases with the help of an \textit{order function} in Definition \ref{defn: order function}. The definitions we make in this section are based on Sections 11.4 and 11.5 in \cite{demaine_geometric_2007} with appropriate modifications and extensions -- for example, we don't require a paper to be orientable, and we allow contact of different parts of a paper, not its folded state.

\begin{defn} \label{defn: generalized}
	We first consider a \textit{connected piecewise-$C^1$ 2-manifold $M$} (defined in Sections 12.3 and 15.2, \cite{zorich_mathematical_2004}). Here at most countable piecewise-$C^1$ curves can be removed from $M$ in such a way that the remainder decomposes into at most countable $C^1$ open 2-manifolds, and we require $M$ to be a closed set. Every point on each piece has a well-defined tangent space, and an Euclidean metric is equipped such that we can measure the length of a piecewise-$C^1$ curve connecting two points on $M$.
	
	A \textit{generalized surface} $g(M)$ is a subset of $\mathbb{R}^3$, where $g: M \rightarrow \mathbb{R}^3$ is a piecewise immersion. A \textit{piecewise immersion} is a continuous and piecewise-$C^1$ function whose derivative is injective on each piece of $M$. Hence $g(M)$ is still connected and a closed set. The \textit{distance} of two points $g(p)$ and $g(q)$ on $g(M)$ ($p,q \in M$) is defined as the infimum of the lengths of $g(\gamma)$, where $\gamma$ is a piecewise-$C^1$ curve connecting $p$ and $q$ on $M$.
\end{defn}

The definition on a generalized surface is an extension of how we usually define a connected piecewise-$C^1$ surface in $\mathbb{R}^3$ (Section 12.2, \cite{zorich_mathematical_2004}). As stated above, Definition \ref{defn: generalized} is still not enough to prevent crossing of different parts of a paper, for which we introduce the definition of \textit{crease pattern} and \textit{order function}.

\begin{defn}
	A \textit{crease pattern} $G$ is a simple graph embedding on a generalized surface $g(M)$, that contains the boundary of those pieces. Each edge of $G$ is a $C^1$ curve on $g(M)$. Note that $\partial g(M) \subset G$. A \textit{crease} is a boundary component of a piece without the endpoints, and we call these endpoints \textit{vertices}. We define a vertex or crease as \textit{inner} if it is not on the boundary $\partial g(M)$, otherwise \textit{outer}; and a piece as \textit{inner} if none of its vertices is on $\partial g(M)$, otherwise \textit{outer}.
\end{defn}

To introduce the order of stacking on different parts of a paper, we need the information of normal vector on $g(M)$. Now we consider a point $p \in M$ and $g(p) \notin G$, there are two coordinate frames of the tangent space (left-handed and right-handed) of $g(p)$, which have opposite orientation. Therefore we can define the orientability and orientation of $g(M)$ in the same way as a piecewise-$C^1$ surface in $\mathbb{R}^3$, which is described in Section 12.3, \cite{zorich_mathematical_2004}. The following definition on order function depends on the orientability of $g(M)$.

\begin{defn} \label{defn: order function}
	If $p,q \in M$, $g(p),g(q) \notin G$ and $g(p)=g(q)$, an \textit{order function} $\lambda$ is defined on $p,q$ that describes the order of stacking of layers locally. If $g(M)$ is orientable, all the tangent spaces have a consistent orientation, then we define $\lambda(p, q)=1$ if $p$ is in the direction pointed to by the normal vector $\boldsymbol{n}_g(q)$ (on the ``top'' side of $q$); and $\lambda(p, q)=-1$ if $p$ is in the direction pointed to by $-\boldsymbol{n}_g(q)$ (on the ``bottom'' side of $q$). If $g(M)$ is not-orientable, since any surface in $\mathbb{R}^3$ is non-orientable if and only if it contains a M\"obius band as a subspace, and a M\"obius band is the disjoint union of two orientable surfaces (Section 12.3, \cite{zorich_mathematical_2004}), we can always divide $g(M)$ into at most countable orientable generalized surfaces. By assigning a unified orientation of these, we can still describe the order of stacking using an order function. The information of how $g(M)$ is divided into orientable pieces should be included in the description of the order function.
\end{defn}

\begin{defn} \label{defn: basic}
	A generalized surface $S$ is a \textit{paper} if we can find a crease pattern $G$, such that $S$ makes the order function $\lambda$ satisfy the four conditions described in Section 11.4 of \cite{demaine_geometric_2007}, which prevent the crossing of paper. 
\end{defn}

\begin{rem}
	We require a generalized surface to be a closed set since it is physically reasonable for a paper to contain its boundary. We require a generalized surface to be connected since otherwise it is the union of at most countable connected generalized surfaces (also called its components), and in that case each component is a paper. Definition \ref{defn: basic} prevents crossing of different parts of a paper when they contact with each other. The contact of a point with a crease point is allowed, but the conditions for order function can not be satisfied if there is a crossing. A paper does not need to be developable or bounded.
\end{rem}

Figure \ref{fig:general papers} shows some examples of objects which are, and are not, papers.

\begin{figure}[!]
	\noindent \begin{centering}
		\includegraphics[width=0.9\linewidth]{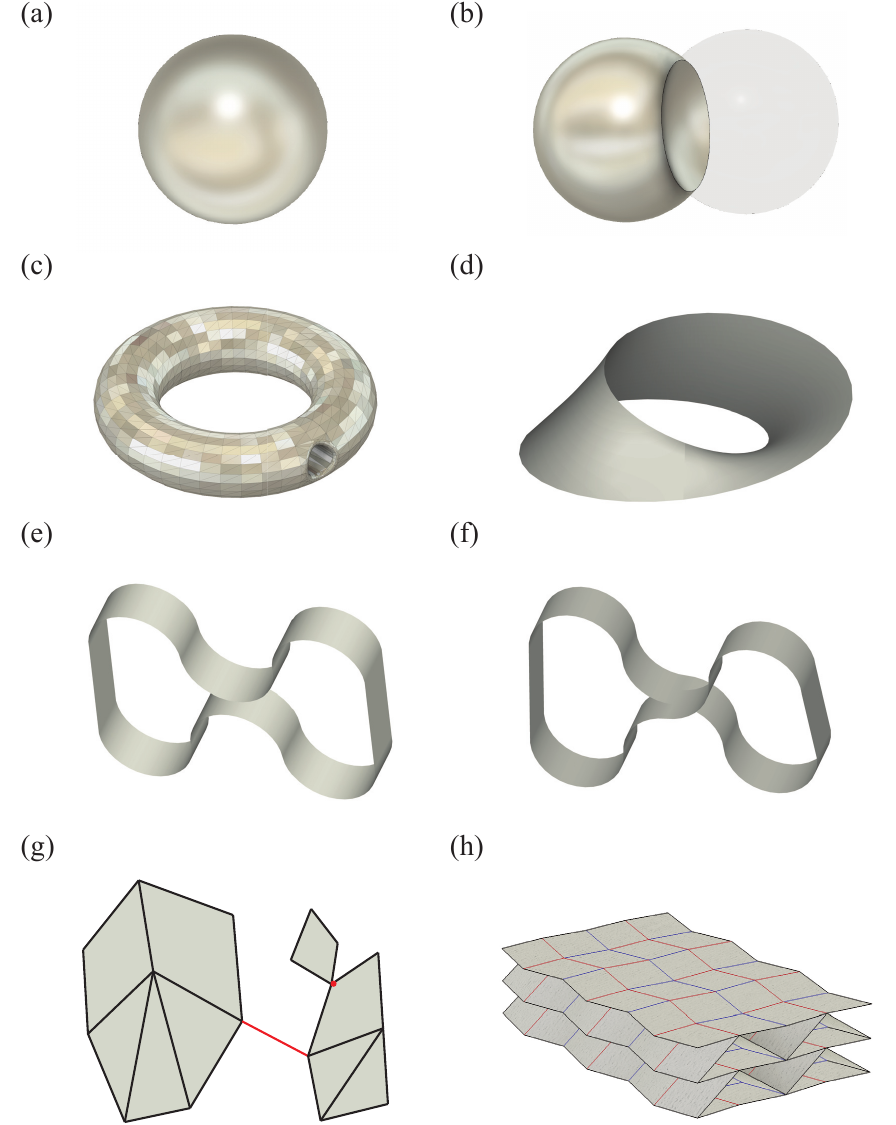}
		\par\end{centering}
	
	\caption{\label{fig:general papers}Examples of objects that do, or do not, conform to the definition of ``paper'' used in this article. (a)--(e) are papers with unusual shapes. (a) is a sphere, which can be regarded as a paper. (b) is a folded state of (a) with a curved crease, generated by the intersection of two identical spheres. (c) is a piecewise-planar paper with two dough-nut holes (Euler Characteristic -2). (d) is a M\"obius band, an example of non-orientable surface, which is the disjoint union of two orientable surfaces. (e) shows the case where there is contact between different parts of a paper. (f)--(h) are ``paper-like'' objects that are not papers even though they might be physically ``foldable''. (f) is similar to (e), but the two layers intersect with each other, hence it does not satisfy the condition on the order function. Part of (g) is 1-dimensional. It can be regarded as a ``creased paper'', whose crease pattern has a cut vertex and a bridge (coloured red), which could be considered as a spherical joint and a bar. (h) is an example of the stacking meta-material, where the contact points are crease points. It is not connected following the requirement of a paper, instead it can be regarded as the union of five papers (plotted by Freeform Origami \cite{tachi_freeform_2010-1}).}
\end{figure}


\begin{defn} \label{defn: folded state}
	Here we define the \textit{folded state} and \textit{folding motion} for a creased paper. We say a \textit{creased paper} $(S,G)$ is an ordered pair of a paper $S$ and a crease pattern $G$. A \textit{folded state} of a creased paper $(S, G)$ can be expressed by a pair $(f, \lambda)$, where $f$ is an isometry function excluding Euclidean motion that maps $(S, G)$ to another creased paper $(f(S), f(G))$ and preserves the distance; $\lambda$ is the order function of $(f(S), f(G))$. A folded state is \textit{free} when the domain of the order function is empty. The identity map with its order function $(I, \lambda)$ is the \textit{trivial} folded state.
	
	A \textit{folding motion} is a family of continuous functions mapping each time $t \in [0,1]$ to a folded state $(f_t, \lambda_t)$. The continuity of $f_t$ is defined under the \textit{supremum metric}, $\forall t \in [0,1]$, $\forall \epsilon>0$, $\exists \delta>0$ s.t. if $|t'-t|<\delta$ 
	\begin{equation}
	\sup_{p \in S} ||f_{t'}(p)-f_t(p)||<\epsilon
	\end{equation}
	If $S$ is orientable, the continuity of $\lambda_t$ with respect to $t$ is described in Section 11.5 of \cite{demaine_geometric_2007}. If $S$ is not orientable, when we define its order function, $S$ is divided into orientable pieces and assigned a unified orientation. Then at each non-crease contact point, the continuity of $\lambda_t$ with respect to $t$ should follow Section 11.5 of \cite{demaine_geometric_2007}.
		
	If there is a folding motion between two different folded states $(f_1, \lambda_1)$ and $(f_2, \lambda_2)$, we say $(f_1, \lambda_1)$ is \textit{foldable} to $(f_2, \lambda_2)$, and the creased paper is \textit{foldable}. If $(f, \lambda)$ is not foldable to any other folded state, we say this folded state is \textit{rigid}.
\end{defn}

\begin{rem}
	Mapping a creased paper to another creased paper requires the isometry function $f$ to be a piecewise immersion.
\end{rem}

Based on the definitions above, now we start to discuss rigid origami. The only difference between origami and rigid origami is the restriction of $f$ on each piece. 

\begin{defn} \label{defn: rigid origami}
	A \textit{rigidly folded state} is a folded state where the restriction of the isometry function $f$ on each piece is a combination of translation and rotation. A \textit{rigid folding motion} is a folding motion where all the folded states are rigidly folded states. 
\end{defn}

\begin{rem}
	We do not include reflection in the isometry function of rigid origami restricted to a piece.
\end{rem} 

From Definition \ref{defn: rigid origami}, since each piece is under a combination of translation and rotation, we can make appropriate simplification on the creased paper $(S, G)$. First, we can require each inner crease to be a straight-line, otherwise the two pieces adjacent to this inner crease will not have relative rigid folding motion, which is not what we expect in rigid origami. Second, if all inner creases are straight-line, there is no essential difference between the rigid folding motion of a planar piece and a general piece, as well as between a straight-line outer crease and a general outer crease. Therefore,

\begin{defn} \label{def: rigid origami}
	In rigid origami, our object of study can be limited to a creased paper $(P,C)$ with a \textit{piecewise-planar} paper $P$ and a \textit{straight-line} crease pattern $C$. Here we call each planar piece a \textit{panel}. The set of all rigidly folded states $\{(f,\lambda)\}_{P,C}$ is called the \textit{rigidly folded state space}.
\end{defn}

As an alternative to Definition \ref{def: rigid origami}, \cite{demaine_non_2011} shows that an isometric map on a creased paper will become piecewise rigid if we require the paper $S$, crease pattern $G$ and isometry function $f$ to have stronger properties. This result is provided in Appendix \ref{app: alternative}. 

In the following section, we will start to discuss the configuration of a creased paper $(P,C)$ in rigid origami.

\section{The Configuration Map of a Creased Paper in Rigid Origami}

In order to study the rigid-foldability between possible rigidly folded states of a creased paper $(P,C)$ in rigid origami, in this section we introduce the configuration map to characterize a rigidly folded state. Before that, we need some preliminary definitions.

\begin{defn} \label{sector and folding angles}
	At every vertex, we define the angles between adjacent creases by \textit{sector angles}, each of which is named $\alpha_{i}$. $\boldsymbol{\alpha}=\{\alpha_{i}\}$ is the set of all sector angles, which we regard as fixed variables under a given creased paper $(P,C)$, satisfying:
	\begin{equation*}
	\alpha_{i} \in (0, 2\pi); ~\textrm{except~that,~for~a~degree-1~vertex}, ~\alpha = 2\pi
	\end{equation*} 
	Then we specify an orientation for $(P,C)$ (If $S$ is not orientable, an unified orientation is assigned in the same way to its division as mentioned in Definition \ref{defn: order function}). At each inner crease, we define a signed \textit{folding angle} $\rho_j$ by which the two panels adjacent to the inner crease deviate from a plane. All the folding angles are measured from the specified orientation. $\boldsymbol{\rho}=\{\rho_j\}$ is the set of all folding angles, satisfying:
	\begin{equation*}
	\rho_j \in [-\pi,\pi]
	\end{equation*}
	The sector and folding angles are also explained graphically in Figure \ref{fig: sector_angles}.
\end{defn}

\begin{figure}[!tb]
	\noindent \begin{centering}
		\includegraphics[width=0.9\linewidth]{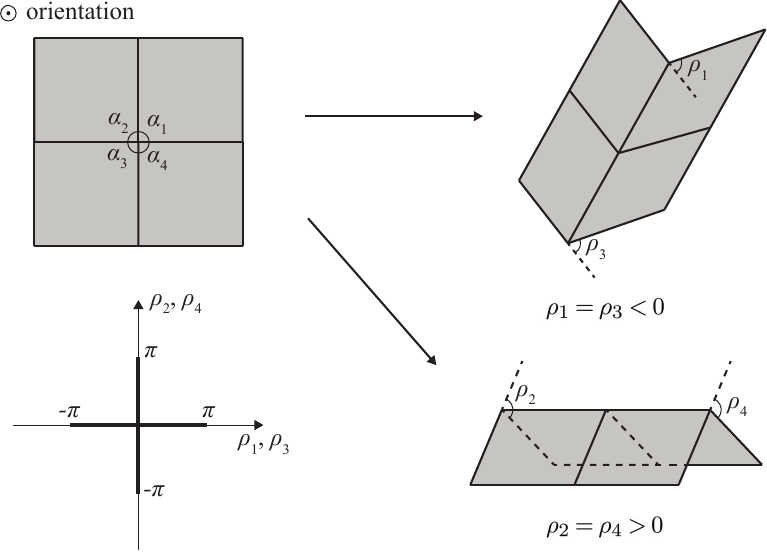}
		\par\end{centering}
	
	\caption{\label{fig: sector_angles}Here we show a simple creased paper with four sector angles $\alpha_1=\alpha_2=\alpha_3=\alpha_4=\pi/2$, and we select its two rigidly folded states with $\rho_1=\rho_3<0$, $\rho_2=\rho_4=0$ and $\rho_2=\rho_4>0$, $\rho_1=\rho_3=0$. The folding angles are measured in the specified orientation, which is the ``top'' side of the paper, facing the readers. The configuration space is a "cross", and the information of stacking is contained in the difference of $\rho_j = \pm \pi$. The configuration map is a bijection from the configuration space to the rigidly folded state space.}
\end{figure}

We introduce the sector and folding angles in order to find an explicit expression for the isometry function $f$ of a rigidly folded state for any point $p \in P$ and a given $\boldsymbol{\rho}$. The set of folding angles of the trivial rigidly folded state is denoted by $\boldsymbol{\rho}_0$, which is not necessarily $\boldsymbol{0}$.

From the analysis above, we know a rigidly folded state $(f, \lambda)$ corresponds with a set of folding angles $\boldsymbol{\rho}$. However, different $(f, \lambda)$ can be mapped to the same $\boldsymbol{\rho}$ --- an example is shown in Figure \ref{fig: order function}. The difference in $\rho_j = \pm \pi$ can only represent the information of order function on two panels adjacent to an inner crease. Therefore we still need the order function when expressing a rigidly folded state with $\boldsymbol{\rho}$.

\begin{figure}[!tb]
	\noindent \begin{centering}
		\includegraphics[width=1\linewidth]{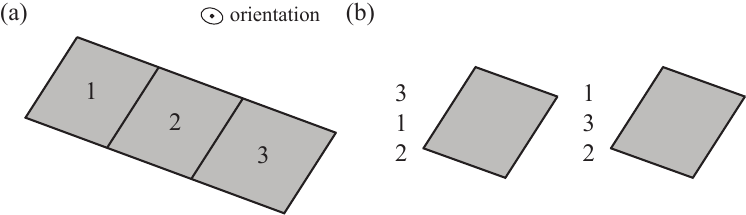}
		\par\end{centering}
	
	\caption{\label{fig: order function}(a) is a creased paper with three identical squares. Here we choose the orientation to be the ``top'' side of the paper, facing the readers. (b) shows two different rigidly folded states of (a) with the same folding angle $\{-\pi,-\pi\}$. The numbers are stacking sequences of the panels. The order functions of these two rigidly folded states are $\lambda(1,3)=1$, $\lambda(3,1)=-1$ and $\lambda(1,3)=-1$, $\lambda(3,1)=1$.}
\end{figure}

\begin{prop} \label{prop explicit def of f}
	Given a creased paper $(P,C)$, we fix a panel $P_0$ to exclude Euclidean motion. Set one of the vertices as the origin and one of its creases which we label $c_0$ as the $x$-axis. then build the right-hand \textit{global coordinate system} with $xy$-plane on this panel. For every $p \in P$, $p = [x,y,z]^T$, we can always draw a path from the origin $(0,0,0)$ to $p$. The path intersects with $C$ on some inner creases which we label $c_k$ ($k \in [1,K]$). The folding angle on crease $c_k$ is $\rho_k$.  
	
	We can also define a \textit{local coordinate system} on panel $P_k$ ($k \in [1,K]$), whose origin $O_k$ is on one endpoint of $c_k$, $x$-axis is on $c_k$ and $z$-axis is normal to the panel. The direction of all $z$-axes of the global and local coordinate systems are consistent with the orientation of the paper and hence consistent with the definition of the sign of folding angles. We specify the direction of the $x$-axis on $c_k$ so that the rotation from panel $P_{k-1}$ to $P_k$ is a rotation $\rho_k$ about that axis. We denote the angle between the $x$-axes of local coordinate systems on creases $c_{k-1}$ and $c_k$ as $\beta_k$. $\beta_k$ is a linear function of the sector angles $\boldsymbol{\alpha}$. $p$ is in the closure of $P_K$.
	
	Now we can write the coordinate of $p$ in the local coordinate system as $f^K(p) = [f_x^K(p),f_y^K(p),f_z^K(p)]^T$ (see Figure \ref{fig:configuration map}). When using homogeneous matrices to represent the transformation from local to global coordinate system along the path, the result is:	
	\begin{equation} 
	\left[ \begin{array}{c}
	\\
	f(p) \\
	\\
	1
	\end{array} \right]=
	\left[ \begin{array}{c}
	f_x(p) \\
	f_y(p) \\
	f_z(p) \\
	1
	\end{array} \right] = T_K(\boldsymbol{\rho})
	\left[ \begin{array}{c}
	f_x^K(p) \\
	f_y^K(p) \\
	f_z^K(p) \\
	1
	\end{array} \right]
	\end{equation}
	where,
	\begin{align} \label{eq: general 1}
	T_K(\boldsymbol{\rho})=\prod_{1}^{K}
	\left[ \begin{array}{cccc}
	\cos \beta_k & -\sin \beta_k & 0 & a_k \\
	\sin \beta_k & \cos \beta_k & 0 & b_k \\
	0 & 0 & 1 & 0   \\
	0 & 0 & 0 & 1
	\end{array} \right]
	\left[ \begin{array}{cccc}
	1 & 0 & 0 & 0 \\
	0 & \cos \rho_k & -\sin \rho_k & 0 \\
	0 & \sin \rho_k & \cos \rho_k & 0 \\
	0 & 0 & 0 & 1
	\end{array} \right]
	\end{align}
	$[a_k,b_k,0]^T$ ($k \in [1,K]$) is the position of $O_k$ in the local coordinate system for panel $P_{k-1}$. The product $T$ is formed by post-multiplication.
	
	Let $\boldsymbol{\rho}=\boldsymbol{\rho}_0$:	
	\begin{equation}
	\left[ \begin{array}{c}
	x \\
	y \\
	z \\
	1
	\end{array} \right] = T_K(\boldsymbol{\rho}_0)
	\left[ \begin{array}{c}
	x^K \\
	y^K \\
	z^K \\
	1
	\end{array} \right]
	\end{equation}
	where $[x^K, y^K, z^K]$ is the coordinate of $p$ in the local coordinate system on panel $K$. As panel $P_K$ moves rigidly, we have
	\begin{equation}
	\left[ \begin{array}{c}
	f_x^K(p) \\
	f_y^K(p) \\
	f_z^K(p) \\
	1
	\end{array} \right]=
	\left[ \begin{array}{c}
	x^K \\
	y^K \\
	z^K \\
	1
	\end{array} \right]
	\end{equation}
	Thus,
	\begin{equation} \label{eq: general 2} 
	\left[ \begin{array}{c}
	f_x(p) \\
	f_y(p) \\
	f_z(p) \\
	1
	\end{array} \right] = T_K(\boldsymbol{\rho})T_K^{-1}(\boldsymbol{\rho}_0)
	\left[ \begin{array}{c}
	x \\
	y \\
	z \\
	1
	\end{array} \right]
	\end{equation}
\end{prop}

\begin{figure}[!tb]
	\noindent \begin{centering}
		\includegraphics[width=0.8\linewidth]{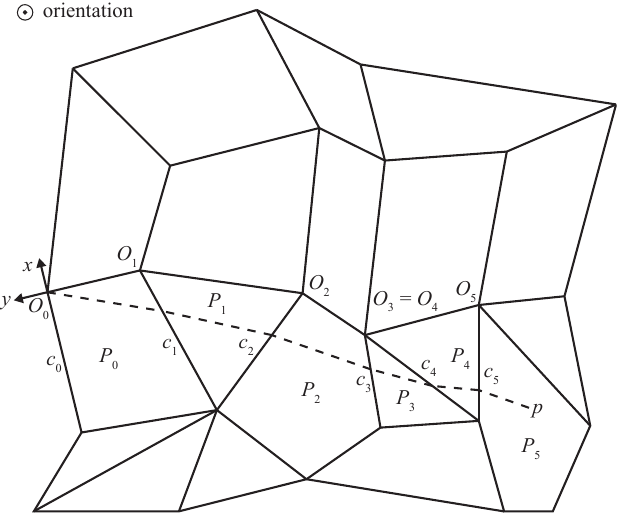}
		\par\end{centering}
	
	\caption{\label{fig:configuration map}A creased paper with one boundary component --- the outer face (Euler characteristic $1$). Here we choose the orientation to be the ``top'' side of the paper, facing the readers. We show the point $p$, intermediate inner creases $c_k$, panels $P_k$ and origins $O_k$ ($k \in [0,5]$).}
\end{figure}

\begin{rem}
	In Definition \ref{defn: order function}, the order function is defined on the non-crease contact points, and from the consistency condition of order function (Section 11.4, \cite{demaine_geometric_2007}), in rigid origami pairs of points on stacked panels have the same order function. Here we note that if just describing the order of a creased paper by the stacking order of panels, it may not be a well-defined order. An example is the classical square twist, where the ordering of some panels $a, b, c, d$ are $a>b>c>d>a$ ($>$ means on the top side of).
\end{rem} 

\begin{defn} \label{configuation bijection}
	 Given a creased paper $(P,C)$, we define the \textit{configuration map} $F: \{(\boldsymbol{\rho}, \lambda)\}_{P,C} \rightarrow \{(f, \lambda)\}_{P,C}$, where $\{\boldsymbol{\rho}\}_{P,C}$ is the set of folding angles of all possible rigidly folded states of $(P,C)$, and $\{\lambda\}_{P,C}$ is the collection of order function for each $\boldsymbol{\rho}$. The collection of this pair $\{(\boldsymbol{\rho}, \lambda)\}_{P,C}$ is called the \textit{configuration space}. An example of the configuration map and configuration space is provided in Figure \ref{fig: sector_angles}. Now $F$ naturally becomes a bijection. The order function $\lambda$ can be a multivalued function of $\boldsymbol{\rho}$ in the configuration space, and when we use $\boldsymbol{\rho}$ to describe the configuration, $\lambda$ does not need to include the stacking of adjacent panels since this information is included in the difference of $\rho_j = \pm \pi$.
\end{defn}	

Before further discussion we want to explain the limit of $\boldsymbol{\rho}$, $f$ and $\lambda$ in the configuration space. A series of $\{\boldsymbol{\rho}_n\} \rightarrow \boldsymbol{\rho}$ is naturally defined; $\{f_n\} \rightarrow f$ means the supreme metric $\sup |f_n(p)-f(p)| \rightarrow 0$; $\{\lambda_n\} \rightarrow \lambda$ requires $\{\lambda_n\}$ to satisfy the continuity condition mentioned in Section 11.5 of \cite{demaine_geometric_2007}. These conditions for $\{\lambda_n\}$ are to guarantee the ``approach'' of $\{\lambda_n\}$ is ``physically admissible''.

\begin{prop} \label{properties of F}
	The configuration map $F$ has the following properties:
	\begin{enumerate}[(1)]
		\item $F$ is \textit{scale-independent}. That means, if we inflate $P$ to $P'$ by a factor $c$ ($g>0$), for any $p'=cp$, $f(p')=cf(p)$; if an order function is defined on $p,q$, and $p'=cp$, $q'=cq$ then $\lambda(p',q')=\lambda(p,q)$.
		\item We extend $f$ to be the function of $\boldsymbol{\rho}$, $\boldsymbol{\rho}_0$ and $p$, then we can formally calculate
		\begin{equation*}
		f(-\boldsymbol{\rho}, -\boldsymbol{\rho}_0, \left[ \begin{array}{c}
		x \\
		y \\
		-z
		\end{array} \right])
		= \left[ \begin{array}{ccc}
		1 & 0 & 0 \\
		0 & 1 & 0 \\
		0 & 0 & -1
		\end{array} \right] f(\boldsymbol{\rho}, \boldsymbol{\rho}_0, \left[ \begin{array}{c}
		x \\
		y \\
		z 
		\end{array} \right])
		\end{equation*}
		Therefore the positions of $f(-\boldsymbol{\rho}, -\boldsymbol{\rho}_0)$ and $f(\boldsymbol{\rho}, \boldsymbol{\rho}_0)$ are symmetric to $P_0$, $\{\boldsymbol{\rho}\}_{P,C}$ is symmetric to $\boldsymbol{0}$, and $\lambda(-\boldsymbol{\rho})=-\lambda(\boldsymbol{\rho})$.
		\item If $F$ is defined on a neighbourhood of a point in $\{(\boldsymbol{\rho},\lambda)\}_{P,C}$, or $F^{-1}$ is defined on a neighbourhood of a point in $\{(f,\lambda)\}_{P,C}$, $F$ is a homeomorphism.
		\item $\{\boldsymbol{\rho}\}_{P,C}$ is discrete or compact.
	\end{enumerate}
\end{prop}

\begin{proof}	
	Statement (1): Inflating $P$ means to keep all the sector angles and inflate the lengths of all the creases by $c$, so for any $p'=cp$, direct calculation gives $f(p')=cf(p)$. Also, inflating does not change the order function.
	
	Statement (2): This expression is equivalent to:	
	\begin{equation} 
	T_K(-\boldsymbol{\rho})T_K^{-1}(-\boldsymbol{\rho}_0)
	\left[ \begin{array}{c}
	x \\
	y \\
	-z \\
	1
	\end{array} \right] = \left[ \begin{array}{cccc}
	1 & 0 & 0 & 0\\
	0 & 1 & 0 & 0\\
	0 & 0 & -1 & 0\\
	0 & 0 & 0 & 1
	\end{array} \right] T_K(\boldsymbol{\rho})T_K^{-1}(\boldsymbol{\rho}_0)
	\left[ \begin{array}{c}
	x \\
	y \\
	z \\
	1
	\end{array} \right]
	\end{equation}		
	which can be proved by induction and direct symbolic calculation. When changing all the folding angles $\boldsymbol{\rho}$ to $-\boldsymbol{\rho}$ and calculate $\lambda$ from the same orientation, the order of panels is reversed, so $\lambda(-\boldsymbol{\rho})=-\lambda(\boldsymbol{\rho})$.
	
	Statement (3): In the neighbourhood of $(\boldsymbol{\rho},\lambda)$ that $F$ is defined, for any series of $(\boldsymbol{\rho}_n,\lambda_n) \rightarrow (\boldsymbol{\rho},\lambda)$, we need to prove $(f_n,\lambda_n) \rightarrow (f,\lambda)$. This is guaranteed because $f$ is smooth with respect to $\boldsymbol{\rho}$. We also need to prove that for any series of $(f_n,\lambda_n) \rightarrow (f,\lambda)$ we have $(\boldsymbol{\rho}_n,\lambda_n) \rightarrow (\boldsymbol{\rho},\lambda)$. This is because if $\sup |f_n(p)-f(p)| \rightarrow 0$, $T_{nK}(\boldsymbol{\rho})-T_K(\boldsymbol{\rho}) \rightarrow 0$, then $\boldsymbol{\rho}_n \rightarrow \boldsymbol{\rho}$. The case for $F^{-1}$ defined on a neighbourhood of a point in $\{(f,\lambda)\}_{P,C}$ is similar.
	
	Statement (4): If $\{\boldsymbol{\rho}\}_{P,C}$ is not discrete, the rigidly folded state space $\{(f, \lambda)\}_{P,C}$ is closed in the sense we define the limit of $f$ and $\lambda$, since the properties of an isometry function $f$ and order function $\lambda$ are preserved under limitation. From statement (3), $\{\boldsymbol{\rho}\}_{P,C}$ is closed. Because $\{\boldsymbol{\rho}\}_{P,C}$ is also bounded, it is compact.
\end{proof}

Since it is more convenient to study the rigid-foldability in the configuration space rather than in the rigidly folded state space, we provide the next conclusion.

\begin{thm} \label{thm for a paper}
	Given a creased paper $(P,C)$, $(f_1, \lambda_1)$ is rigid-foldable to $(f_2, \lambda_2)$ if and only if $(\boldsymbol{\rho}_1,\lambda_1)$ is 0-connected to $(\boldsymbol{\rho}_2,\lambda_2)$ in the configuration space $\{(\boldsymbol{\rho},\lambda)\}_{P,C}$.
\end{thm}

\begin{proof}
	The following proof is an extension of ``The combination of continuous functions is continuous.''	
	
	Sufficiency: If $(\boldsymbol{\rho}_1,\lambda_1)$ and $(\boldsymbol{\rho}_2,\lambda_2)$ are 0-connected in $\{(\boldsymbol{\rho},\lambda)\}_{P,C}$, we parametrize this path by $L: t \in [0,1] \rightarrow \{(\boldsymbol{\rho},\lambda)\}_{P,C}$. From statement (3) in Proposition \ref{properties of F}, on this path $L$, the configuration map $F: \{(\boldsymbol{\rho},\lambda)\}_{P,C} \rightarrow \{(f, \lambda)\}_{P,C}$ is continuous. It can be directly verified that the composite map $F \circ L: t \in [0,1] \rightarrow \{(f, \lambda)\}_{P,C}$ is also continuous. Therefore $(f_1, \lambda_1)$ is rigid-foldable to $(f_2, \lambda_2)$.
	
	Necessity: $(f_1, \lambda_1)$ is rigid-foldable to $(f_2, \lambda_2)$ means there exists a path in this function space $L' : t \in [0,1] \rightarrow \{(f, \lambda)\}_{P,C}$. Every point on this path corresponds with a point in the configuration space $\{(\boldsymbol{\rho},\lambda)\}_{P,C}$, and from statement (6) in Proposition \ref{properties of F}, the inverse of configuration map $F^{-1} : \{(f, \lambda)\}_{P,C} \rightarrow \{(\boldsymbol{\rho},\lambda)\}_{P,C}$ is continuous on $L'$. Similarly we can verify that the composite map $F^{-1} \circ L': t \in [0,1] \rightarrow \{(\boldsymbol{\rho},\lambda)\}_{P,C}$ is also continuous. Therefore $(\boldsymbol{\rho}_1,\lambda_1)$ and $(\boldsymbol{\rho}_2,\lambda_2)$ are 0-connected.
\end{proof}

\begin{defn} \label{defn: global}
If the configuration space is a collection of discrete points, we say it is 0-\textit{dimensional}. If the configuration space is $(\boldsymbol{0},\emptyset)$ or a collection of two points $(\boldsymbol{\rho},\lambda)$ and $(-\boldsymbol{\rho},-\lambda)$, we say it is \textit{trivial}, and this creased paper is \textit{globally rigid}.
\end{defn}

From Theorem \ref{thm for a paper}, the existence of non-trivial rigidly folded states and rigid-foldability are the shape and 0-connectedness of the configuration space $\{(\boldsymbol{\rho},\lambda)\}_{P,C}$, which is not easily characterized. We will mainly present three methods to study the configuration space: algebraic, geometric and numeric methods. Usually we focus on $\{\boldsymbol{\rho}\}_{P,C}$ and then check $\lambda$ when there are multiple $\lambda$ for a particular $\boldsymbol{\rho}$. For convenience from now on we use $(\boldsymbol{\rho},\lambda)$ to represent a rigidly folded state.

\section{Algebraic Analysis of the Configuration Space} \label{section: algebraic}

The algebraic method is to analyse possible position of panels around vertices (equation \eqref{eq: consistency 1}) and holes (equation \eqref{eq: consistency 2}) symbolically, then remove the solutions that do not satisfy the boundary constraints described below. Before further discussion we need some definitions.

\begin{defn} \label{def constraints}
	Given a creased paper $(P,C)$, $W_{P,C}$ is the solution space of the \textit{consistency constraints} given in equations \eqref{eq: consistency 1} and \eqref{eq: consistency 2} where every $\rho_j \in [-\pi, \pi]$. How they are derived for a developable creased paper is provided in \cite{tachi_rigid_2015}, which can also be applied to a general creased paper.
	\begin{enumerate} [(1)]
		\item At every inner degree-$n$ vertex: (see Figure \ref{fig: single creased papers}(a))
		\begin{equation} \label{eq: consistency 1}
		T_n(\boldsymbol{\rho}) = I
		\end{equation}
		where,
		\begin{align*} 
		T_n(\boldsymbol{\rho})=\prod_{1}^{n}
		\left[ \begin{array}{ccc}
		\cos \alpha_{j} & -\sin \alpha_{j} & 0\\
		\sin \alpha_{j} &  \cos \alpha_{j} & 0\\
		0                &                 0 & 1
		\end{array} \right]
		\left[ \begin{array}{ccc}
		1                &                 0 & 0             \\
		0                &     \cos \rho_{j} & -\sin \rho_{j}\\
		0                &     \sin \rho_{j} &  \cos \rho_{j}
		\end{array} \right]
		\end{align*}
		$\alpha_{j}$ is between $c_{j-1}$ and $c_j$ ($j \in [2, n]$). $\alpha_{1}$ is between $c_n$ and $c_1$. $R$ is formed by post-multiplication. Equation \eqref{eq: consistency 1} can be derived by following Proposition \ref{prop explicit def of f} and choosing the path to be the one shown in Figure \ref{fig: single creased papers}(a). Only three of the nine equations are independent, for instance, the diagonal elements.
		\item Suppose there are $h$ holes (If the Euler Characteristic is $2$ or $1$, there is no such constraint). For a hole with $n$ inner creases (see Figure \ref{fig: single creased papers}(b)), called a \textit{degree-$n$} hole:
		\begin{equation} \label{eq: consistency 2}
		T_n(\boldsymbol{\rho}) = I
		\end{equation}
		where,
		\begin{align*} 
		T_n(\boldsymbol{\rho})=\prod_{1}^{n}
		\left[ \begin{array}{cccc}
		\cos \beta_j & -\sin \beta_j & 0 & a_j \\
		\sin \beta_j & \cos \beta_j & 0 & b_j \\
		0 & 0 & 1 & 0   \\
		0 & 0 & 0 & 1
		\end{array} \right]
		\left[ \begin{array}{cccc}
		1 & 0 & 0 & 0 \\
		0 & \cos \rho_j & -\sin \rho_j & 0 \\
		0 & \sin \rho_j & \cos \rho_j & 0 \\
		0 & 0 & 0 & 1
		\end{array} \right]
		\end{align*}
		Equation \eqref{eq: consistency 2} can be derived by following Proposition \ref{prop explicit def of f} and choosing the path to be the one shown in Figure \ref{fig: single creased papers}(b). $T$ is formed by post-multiplication.  Only six of the sixteen equations are independent. Three of them are in the top left $3 \times 3$ rotation matrix, the other three are the elements from row 1 to row 3 in column 4, which are automatically satisfied if the inner creases are concurrent.
	\end{enumerate}	
	
	The consistency constraints may not include every folding angle, so we define $\widetilde{W}_{P,C}$ as the \textit{extension} of the solution space $W_{P,C}$ to include the folding angles not mentioned in $W_{P,C}$, also with range $[-\pi, \pi]$. $N_{P,C}$ is the collection of all the solutions that do not satisfy the conditions for order function $\lambda$, i.e.\ panels self-intersect, which are called the \textit{boundary constraints} because they are unilateral constraints that only contribute to the boundary of $\{\boldsymbol{\rho}\}_{P,C}$. Some examples have been mentioned in \cite{hull_modelling_2002}.
\end{defn}

\begin{rem}
	Although the lengths of creases may not be involved in the consistency constraints, they are important in the boundary constraints. If the interior of crease pattern is not a forest (see Section \ref{subsection: forest}), the consistency constraints around different vertices or holes are not independent.
\end{rem}

\begin{figure}[!tb]
	\noindent \begin{centering}
		\includegraphics[width=1\linewidth]{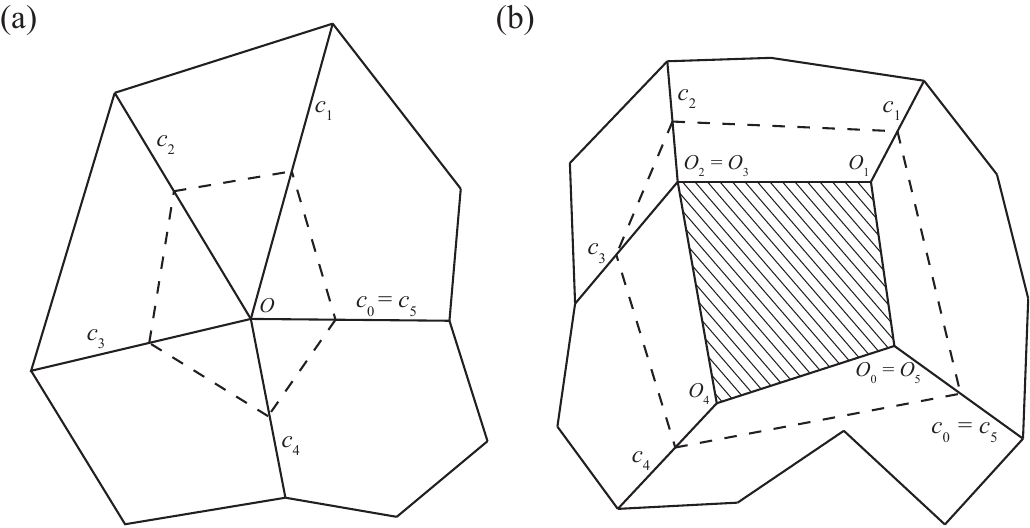}
		\par\end{centering}
	
	\caption{\label{fig: single creased papers}(a) is a degree-5 single-vertex creased paper. (b) is a degree-5 single-hole creased paper, where the shaded area is a hole. For each paper, a path is shown to illustrate the consistency constraints in Definition \ref{def constraints}. We label intermediate inner creases $c_k$ and origins $O_k$ ($k \in [0,5]$). Note that (a) has one boundary component and (b) has two boundary components. We don't need to consider the consistency constraint around the outer boundary because it is naturally satisfied (see Definition \ref{def constraints}).}
\end{figure}

\begin{thm} \label{thm configuration space}
	\cite{tachi_rigid_2015, hull_modelling_2002}
	\begin{equation} 
	\{\boldsymbol{\rho}\}_{P,C} = \widetilde{W}_{P,C} \backslash N_{P,C}
	\end{equation} 
\end{thm}

\begin{cor}
	Some properties of $W_{P,C}$ and $N_{P,C}$.
	\begin{enumerate}[(1)]
		\item $W_{P,C}$ is symmetric to $\boldsymbol{0}$, so $N_{P,C}$ is symmetric to $\boldsymbol{0}$.
		\item $W_{P,C}$ is discrete or compact, so $N_{P,C}$ is open and bounded if not discrete.
	\end{enumerate}
\end{cor}

\begin{proof}
	Statement (1) can be proved by comparing the expressions of $T_n(\boldsymbol{\rho})+T_n(-\boldsymbol{\rho})$ and $T_n(\boldsymbol{\rho})-T_n(-\boldsymbol{\rho})$. With induction and direct symbolic calculation we will know if $T_n(\boldsymbol{\rho})=0$, $T_n(\boldsymbol{-\rho})=0$. Because $\{\boldsymbol{\rho}\}_{P,C}$ is also symmetric to $\boldsymbol{0}$, $N_{P,C}$ is symmetric to $\boldsymbol{0}$. Statement (2) is satisfied because if $W_{P,C}$ is not a discrete set, any limit point of $W_{P,C}$ is also a solution, which means $W_{P,C}$ is closed, also, $W_{P,C}$ is bounded.
\end{proof}

From Definition \ref{def constraints}, the consistency constraints are derived from the consistency of panels around an inner vertex or a hole, therefore we define:  

\begin{defn}
	Given a creased paper $(P,C)$, for each inner vertex $v$, the restriction of $(P,C)$ on all panels incident to $v$ forms a \textit{single-vertex creased paper}. $v$ is the \textit{centre vertex}. Similarly, consider a hole with boundary $h$, whose incident inner creases are not concurrent. The restriction of $(P,C)$ on all panels incident to $h$ forms a \textit{single-hole creased paper} (see Figure \ref{fig: single creased papers}). If the inner creases incident to a hole are concurrent, we still regard it as a single-vertex creased paper, whose centre vertex is the intersection of its inner creases. A single-vertex or single-hole creased paper is called a \textit{single creased paper}.
\end{defn}

\begin{cor} \label{cor local to global}
	$\widetilde{W}_{P,C}$ is the intersection of extensions of the solution spaces of all single creased papers.	
\end{cor}

Corollary \ref{cor local to global} clarifies the link between global and local rigid-foldability, and explains why local rigid-foldability cannot guarantee global rigid-foldability, since even the intersection of 0-connected spaces is not necessarily 0-connected. 

Next we want to apply results in classical rigidity theory to rigid origami, based on those constraints on the configuration space. The collection of equations \eqref{eq: consistency 1} and \eqref{eq: consistency 2} is a system of trigonometric equations $\boldsymbol{A}(\boldsymbol{\rho})=\boldsymbol{0}$. By using the \textit{normalized folding angles} $\boldsymbol{t} : t_j = \tan \frac{\rho_j}{2}$ ($t_j \in [-\infty,\infty]$), the consistency constraints can be written as a system of polynomial equations $\boldsymbol{A}(\boldsymbol{t})=\boldsymbol{0}$ with:

\begin{equation} \label{eq: polynomial}
\cos \rho_j = \dfrac{1-t_j^2}{1+t_j^2}, \quad \sin \rho_j = \dfrac{2t_j}{1+t_j^2}
\end{equation}

For convenience, in this section we will mainly use the above representation $\boldsymbol{A}(\boldsymbol{t})=\boldsymbol{0}$ for further analysis. The degree for each folding angle in $\boldsymbol{A}(\boldsymbol{t})=\boldsymbol{0}$ is 2.

\subsection{Equivalent Definitions on Rigid-foldability}

Here we will follow the idea that is commonly understood for bar-joint frameworks, and give several equivalent definitions on rigid-foldability, which also holds if expressed by the folding angles $\boldsymbol{\rho}$.

\begin{defn}
	An \textit{analytic path} $\gamma: [s_1, s_2] \ni s \rightarrow (\boldsymbol{t}, \lambda)_{P,C}$ that joins $(\boldsymbol{t}_1,\lambda_1)$ and $(\boldsymbol{t}_2,\lambda_2)$ can be expressed as:
	\begin{equation} \label{analytic curve}
	t_i(s)=\sum_{n=0}^{\infty}a_{in}(s-s_1)^n
	\end{equation}
	where $t_i$ are the components of $\gamma$ and $a_{in}$ are the coefficients for the power series of $t_i$.
\end{defn}

\begin{thm}
	Given a creased paper $(P,C)$, the three following definitions on rigid-foldability are equivalent.
	
	\begin{enumerate} [(1)]
		\item (Analytic) Given a rigidly folded state ($\boldsymbol{t}$, $\lambda$), if there exists another rigidly folded state ($\boldsymbol{t}'$, $\lambda'$), where $\boldsymbol{t}'$ is connected to $\boldsymbol{t}$ in the folding angle space $\{\boldsymbol{t}\}_{P,C}$ by an analytic path, and $\lambda'$ is 0-connected to $\lambda$, then ($\boldsymbol{t}$, $\lambda$) is rigid-foldable to ($\boldsymbol{t}'$, $\lambda'$).
		\item (Continuous) Given a rigidly folded state ($\boldsymbol{t}$, $\lambda$), if there exists another rigidly folded state ($\boldsymbol{t}'$, $\lambda'$) which is 0-connected to ($\boldsymbol{t}$, $\lambda$) in the configuration space $(\boldsymbol{t}, \lambda)_{P,C}$, then ($\boldsymbol{t}$, $\lambda$) is rigid-foldable to ($\boldsymbol{t}'$, $\lambda'$).
		\item (Topological) Given a rigidly folded state ($\boldsymbol{t}$, $\lambda$), $\forall \epsilon>0$, if there exists another rigidly folded state ($\boldsymbol{t}_\epsilon$, $\lambda_\epsilon$), s.t. $0<|\boldsymbol{t}_\epsilon-\boldsymbol{t}|<\epsilon$, and all $\{\lambda_\epsilon\}$ satisfy the continuity condition mentioned in Section 11.5 of \cite{demaine_geometric_2007}, then ($\boldsymbol{t}$, $\lambda$) is rigid-foldable.
	\end{enumerate}
\end{thm}


\begin{proof}
	(1) $\rightarrow$ (2) and (2) $\rightarrow$ (3) are natural. Next we prove (3) $\rightarrow$ (1). If from (3), ($\boldsymbol{t}$, $\lambda$) is rigid-foldable, then $\boldsymbol{t}$ is a limit point of the solution of $\boldsymbol{A}(\boldsymbol{t})=\boldsymbol{0}$. From the curve selection lemma (Section 3, \cite{milnor_singular_1968}), there exists a real analytic curve $\gamma$ starting from $\boldsymbol{t}$ and ending at another point $\boldsymbol{t}'$. The continuity of order function is naturally satisfied since the path in (1) is selected from (3).
\end{proof}

\subsection{Analysis from the Tangent Space}

If a creased paper $(P,C)$ has $i$ inner vertices, $j$ inner creases and $h$ holes, the number of equations in $\boldsymbol{A}$ is $3i+6h$, while the number of variables are $j$. We define $\boldsymbol{A}: \mathbb{R}^j \rightarrow \mathbb{R}^{3i+6h}$, then $\boldsymbol{A}$ is smooth (Unlike if the domain were $\{\boldsymbol{t}\}_{P,C}$, when $\boldsymbol{A}$ may not have definition in the neighborhood of a point), and write the Jacobian of $\boldsymbol{A}$ as $\mathrm{J}\boldsymbol{A}= \frac{\rm{d} \boldsymbol{A}}{\rm{d} \boldsymbol{t}}$, whose size is $(3i+6h) \times j$. The expression of $\mathrm{J}\boldsymbol{A}$ with respect to $\boldsymbol{\rho}$ is given in \cite{demaine_zero-area_2016}. If $\rm{rank}(\mathrm{J}\boldsymbol{A}) = \min (3\mathit{i}+6\mathit{h},\mathit{j})$, i.e. reaches its maximum, we say this rigidly folded state is \textit{regular}, otherwise it is \textit{special}.

For every point $\boldsymbol{t} \in \{\boldsymbol{t}\}_{P,C}$, if we denote the degree of freedom by $\deg(\boldsymbol{t})= j-\rm{rank}(\mathrm{J}\boldsymbol{A})$, $\{\boldsymbol{t}\}_{P,C}$ is locally a $\deg(\boldsymbol{t})$ dimensional smooth manifold. It also has a $\deg(\boldsymbol{t})$ dimensional tangent space that can be represented as $\mathrm{J}\boldsymbol{A} \cdot \mathrm{d}\boldsymbol{t}= \boldsymbol{0}$ in $\mathbb{R}^j$. A vector $\mathrm{d}\boldsymbol{t}$ in the tangent space $T_{\boldsymbol{t}}$ is called a \textit{first-order flex}. If $\deg (\boldsymbol{t})=0$, this rigidly folded state $(\boldsymbol{t},\lambda)$ is \textit{first-order rigid}, otherwise \textit{first-order rigid-foldable}. For a regular point, the counting of degree of freedom is provided in \cite{tachi_geometric_2010}.

\begin{thm} \label{thm: first-order}
	\begin{enumerate} [(1)]
		\item If $(\boldsymbol{t},\lambda)$ is first-order rigid, it is rigid. Equivalently, if $(\boldsymbol{t},\lambda)$ is rigid-foldable, it is first-order rigid-foldable.
		\item If $(\boldsymbol{t},\lambda)$ is regular and rigid, it is first-order rigid. Equivalently, if $(\boldsymbol{t},\lambda)$ is regular and first-order rigid-foldable, it is rigid-foldable.
	\end{enumerate}
\end{thm}

\begin{proof}
	Statement (1): If for a given point $\boldsymbol{t}$, $\boldsymbol{A}(\boldsymbol{t})=\boldsymbol{0}$ and $\deg (\boldsymbol{t})=0$, from the Implicit Function Theorem (Section 8.5, \cite{zorich_mathematical_2004}), $\boldsymbol{A}^{-1}(\boldsymbol{0})$ is a 0-dimensional manifold in some neighborhood of $\boldsymbol{t}$, hence $(\boldsymbol{t},\lambda)$ is rigid. 
	
	Statement (2): If $\boldsymbol{t}$ is a regular point, from the Implicit Function Theorem, the dimension of $\boldsymbol{A}^{-1}(\boldsymbol{0})$ in a neighborhood of $\boldsymbol{t}$ equals to $\deg (\boldsymbol{t})$. Therefore if $(\boldsymbol{t},\lambda)$ is regular, $(\boldsymbol{t},\lambda)$ is rigid if and only if $(\boldsymbol{t},\lambda)$ is first-order rigid.
\end{proof}

\begin{rem}
	For a smooth creased paper (the origami case), similarly we can define the foldability and rigidity in the topological or continuous sense. However, here the first-order rigidity does not imply rigidity, while analytical first-order rigidity implies analytical rigidity, which means that these definitions on foldability are not equivalent. For details please refer to Chapter 12, \cite{spivak_comprehensive_1999}.
\end{rem}

On the other hand, the tangent space of the transpose of Jacobian can be represented as $\mathrm{d}\boldsymbol{\sigma} \cdot \mathrm{J}\boldsymbol{A} =\boldsymbol{0}$, which implies the inner stress local to each vertex, and the inner stress and moments local to each hole. This is a parallel concept of ``self-stress'' for bar-joint frameworks.

Note that unlike a bar-joint framework, projective and affine transformation does not preserve the first-order rigidity or first-order rigid-foldability.

\subsection{Generic Rigid-foldability and Equivalent Panel-hinge Framework}

When studying the rigid-foldability of creased papers, we find that the some of them with a crease pattern isomorphic a particular graph might be (almost) always rigid, or might be (almost) always rigid-foldable. This is the combinatorial property of the crease pattern, which is called the generic rigidity (or rigid-foldability). Here we discuss the theory of generic rigidity, starting from a discussion of first-order rigidity. This is parallel to the work on the generic rigidity of bar-joint frameworks \cite{tay_chapter_1987}.

\begin{defn}
	Given a graph $H$, we define $(P,C)_H$ as a creased paper with a crease pattern isomorphic to $H$, which is also called a \textit{realization} of $H$. $\{(P,C)\}_H$ is the collection of such creased papers, and $\{\boldsymbol{t}\}_H$ is the collection of folding angles of such creased papers.
\end{defn}

\begin{thm} \label{thm: generic rigidity}
	\begin{enumerate}[(1)]
		\item Either (a) the set $X=\{\boldsymbol{t}|(P,C)_H$ is first-order rigid$\}$ is an open dense subset in $\{\boldsymbol{t}\}_H$, and (almost) all realizations of $H$ are first-order rigid; or (b) $X=\emptyset$ and all realizations of $H$ are first-order rigid-foldable.
		\item Either (a) the set $Y=\{\boldsymbol{t}|(P,C)_H$ is rigid$\}$ contains an open dense subset in $\{\boldsymbol{t}\}_H$, and (almost) all realizations of $H$ are rigid; or (b) $Y^c=\{\boldsymbol{t}|(P,C)_H$ is rigid-foldable$\}$ contains an open dense subset in $\{\boldsymbol{t}\}_H$, and (almost) all realizations of $H$ are rigid-foldable.
	\end{enumerate}	
\end{thm}

\begin{proof}
	Statement (1): Assume that there is a first-order rigid realization $(\boldsymbol{t}, \lambda)$ for a given graph $H$, we know that $\deg(\boldsymbol{t})=0$, or equivalently, some minor matrix of $\mathrm{J}\boldsymbol{A}$ has a non-zero determinant. This is a polynomial constraint over $\boldsymbol{t}$. If such a polynomial is non-zero at $\boldsymbol{t}$, it is non-zero in an open dense subset of $\{\boldsymbol{t}\}_H$ (see Section 3 of \cite{bochnak_real_2013}).
	
	Statement (2): Assume that there is a first-order rigid realization $(\boldsymbol{t}, \lambda)$ for a given graph $H$, from Theorem \ref{thm: first-order}, $Y \supset X$, so the first part holds. Then assume there is no first-order rigid realization, $Y^c$ contains the set of all regular realizations of $H$, which is defined as $R=\{\boldsymbol{t}|$all sub-matrices of $\mathrm{J}\boldsymbol{A}$ has non-zero determinant$\}$. This is also a polynomial constraint over $\boldsymbol{t}$. If such a polynomial is non-zero at $\boldsymbol{t}$, it is non-zero in an open dense subset of $\{\boldsymbol{t}\}_H$.
\end{proof}

\begin{defn}
From Theorem \ref{thm: generic rigidity}, we provide several equivalent definitions on the \textit{generic rigidity} of $H$. If $H$ is not generically rigid, it is \textit{generically rigid-foldable}.
	\begin{enumerate} [(1)]
		\item $H$ is generically rigid if $X$ is dense in $\{\boldsymbol{t}\}_H$, or equivalently, (almost) all realizations of $H$ are first-order rigid.
		\item $H$ is generically rigid if there is a first-order rigid realization.
		\item $H$ is generically rigid if there is a first-order rigid regular realization.
	\end{enumerate}

A \textit{generic realization} of $H$ is a rigid realization if $H$ is generically rigid, or is a rigid-foldable realization if $H$ is generically rigid-foldable.  

\end{defn}

\begin{rem}
	For a bar-joint framework, "generic realization" usually means the coordinates of its vertices do not satisfy a system of polynomial equations with rational coefficients, which is a strong sufficient condition for genericity. However, in rigid origami we use folding angles to describe the position of a creased paper, and we haven't found a similar statement. How to exactly determine whether a realization is generic is still unclear in both classical rigidity theory and rigid origami. Sometimes estimating the generic rigidity from the number of constraints, such as saying "$H$ is generically rigid if $3i+6h \ge j$", will fail for a similar reason to the ``double banana'' model in bar-joint frameworks. Often, finding a non-generic rigid-foldable realization for a generically rigid crease pattern is an important topic, where tools such as symmetry might be useful. For example, the Miura-ori is a non-generic realization of a quadrilateral crease pattern, that is generically rigid.
\end{rem}
	
A creased paper $(P,C)$ can be regarded as a panel-hinge framework in rigidity theory, if each panel is substituted by a bar framework with complete graph. Here is a conclusion connecting the property of $H$ and the generic rigidity of $H$.

\begin{defn}
	Given a graph $H$, its dual graph $H^*$ is a graph that has a vertex for each face of $H$, and has an edge whenever two faces of $H$ are separated from each other by an edge. A multigraph $kH$ ($k \in \mathbb{Z}$) is defined as replacing each edge of $H$ by $k$ parallel edges. A \textit{spanning tree} of $H$ is a subset of $H$, which has all the vertices covered with minimum possible number of edges.
\end{defn}

\begin{thm}
	$H$ is generically rigid if and only if $5H^*$ contains 6 edge-disjoint spanning trees \cite{katoh_proof_2011}.
\end{thm}


As stated in Definition \ref{defn: global}, a creased paper $(P,C)$ is \textit{globally rigid} if the configuration space is $(\boldsymbol{0},\emptyset)$ or a collection of two points $(\boldsymbol{\rho},\lambda)$ and $(-\boldsymbol{\rho},-\lambda)$. Global rigidity is a very important concept in classic rigidity theory, but there is no complete result on the generic global rigidity for a panel-hinge framework. It would be very appealing to develop a parallel theory that connects the property of $H$ and the generic global rigidity of $H$.

\subsection{High-order Rigid-foldability}

We want the concept of high-order rigid-foldability to be a generalization of the first-order rigid-foldability. It is natural to consider the Taylor expansion on each equation of $\boldsymbol{A}(\boldsymbol{t})=\boldsymbol{0}$, and the $n$-th order flex should then satisfy the equation given by the first $n$ terms of Taylor expansion. The problem is, as stated before, the degree for each folding angle in $\boldsymbol{A}(\boldsymbol{t})=\boldsymbol{0}$ is 2, so only the concept of first-order flex is sensible if we define high-order rigidity in this way.

Therefore we use the folding angle expression to describe high-order rigidity. The $n$-th order flex should satisfy the equation given by the first $n$ terms of Taylor expansion of $\boldsymbol{A}(\boldsymbol{\rho})=\boldsymbol{0}$ near a point in the configuration space. If the solution space of the $n$-th order flex is $\boldsymbol{0}$, we say the creased paper is \textit{$n$-th order rigid}, otherwise \textit{$n$-th order rigid-foldable}. \cite{demaine_zero-area_2016} gives the explicit constraints of the first-order and second-order flex. 

%
%


\subsection{Duality and Isomorphism of the Tangent Space, Figure Method}

From the constraints on the first-order and second-order flex, \cite{demaine_zero-area_2016} also gives the following conclusions on a developable creased paper, which includes the discussion on reciprocal figure of the crease pattern. We will introduce related definitions first.

\begin{defn}
	Given a planar crease pattern $C$, the \textit{dual graph} $C^*$ is a planar graph whose faces, edges, and vertices correspond to the vertices, edges and faces of $C$, respectively. The \textit{reciprocal figure} $R$ of $C$ is a mapping (with potential self-intersection) of $C^*$ on to the plane with the following properties: 	
	\begin{enumerate} [(1)]
		\item The edge $c^*$ in $R$ mapped from a crease $c$ in $C$ is a line segment perpendicular to $c$.
		\item The face loop is defined for each face $v^*$ in $R$ corresponding to vertex $v$ in $C$ as the sequence of dual edges corresponding to the edges in counterclockwise order around $v$.
		\item Each edge $c^*$ of the reciprocal diagram has assigned a sign $\sigma_i$, such that along any face loop of $v^*$, the direction of the edge $c^*$ is $90^{\circ}$ clockwise or counterclockwise rotation of the vector along the original crease $c$ from corresponding vertex $v$ if the sign is plus or minus, respectively.
	\end{enumerate}
	A \textit{zero-area} reciprocal diagram of $C$ is a reciprocal diagram of $C$ where the signed area of each face in the counterclockwise orientation is zero. The example figures are shown in \cite{demaine_zero-area_2016}. 
\end{defn}

\begin{thm}
	Main results in \cite{demaine_zero-area_2016}.
	\begin{enumerate} [(1)]
		\item A developable creased paper $(P,C)$ is first-order rigid-foldable if and only if there exists a non-trivial (not a point) reciprocal figure of $C$.
		\item A developable creased paper $(P,C)$ is second-order rigid-foldable if and only if there exists a non-trivial zero-area reciprocal figure of $C$.
		\item For a developable single-vertex creased paper, the second-order rigid-foldability is equivalent to rigid-foldability.
		\item For a developable and flat-foldable quadrilateral creased paper \cite{tachi_generalization_2009}, the second-order rigid-foldability is equivalent to rigid-foldability.
	\end{enumerate}	
\end{thm}

\begin{rem}
	Statements (3) and (4) imply a very interesting topic, which is to find the equivalence of the second-order rigid-foldability (or possibly higher-order) to rigid-foldability for some special creased papers. We believe that there must be some deeper reason to justify the equivalence of high-order rigid-foldability and rigid-foldability for a certain type of creased paper, possibly from real algebraic geometry, and we will try to discuss it in a future article.
\end{rem}

Apart from the analyses above, we can also find some isomorphism on the tangent space $T_{\boldsymbol{\rho}}$ for a developable creased paper. From the constraints on the first-order flex, $T_{\boldsymbol{\rho}}$ is isomorphic to the space of the magnitude of admissible axial forces, when regarding the inner creases as rigid bars, and the holes as rigid panels. It means that $\mathrm{d} \boldsymbol{\rho} \in T_{\boldsymbol{\rho}}$ if and only if the above model is in equilibrium when seeing $\mathrm{d} \boldsymbol{\rho}$ as the magnitude of corresponding axial forces \cite{tachi_design_2012}.

\begin{prop}
	A developable creased paper is first-order rigid-foldable if the degree of each inner vertex is no less than 4.
\end{prop}

\begin{proof}
	We start from assuming that there is no hole in this developable creased paper (Euler Characteristic 1), which has $i$ inner vertices and $j$ inner creases. From the isomorphism mentioned above, for this creased paper $\deg (\boldsymbol{0})=j-2i$. If the degree of each inner vertex is 4 and the creased paper is unbounded, $j=2i$. However, consider the effect of boundary, each inner vertex share more inner creases, so $j>2i$, and if there are some vertices whose degree is greater than 4, $j$ will be even greater, therefore $\deg (\boldsymbol{0})=j-2i>0$. For a developable creased paper with holes, it can be generated by cutting from a developable creased paper without hole, hence it is still first-order rigid-foldable. 
\end{proof}

\cite{schief_integrability_2008} gives another way to describe the constraints on the first-order flex by the existence of dual graph. If we assume a point $\boldsymbol{\rho}$ is differentiable with respect to a parameter $t$ (only need the form of derivative) in $\{\boldsymbol{\rho}\}_{P,C}$, then for every $K$, 
\begin{equation}
\dot{\rho_K} \overrightarrow{c_K}= \overrightarrow{\omega_{K}}-\overrightarrow{\omega_{K-1}}
\end{equation}
where $\overrightarrow {c_K}$ is the direction vector of crease $c_K$ where the folding angle is $\rho_K$, and $\omega_{K}$ is the angular velocity of panel $K$ in the global coordinate system. From Proposition \ref{prop explicit def of f} we obtain:
\begin{equation}
\overrightarrow{\omega_K}(\boldsymbol{\rho})_\times = \dot T_K(\boldsymbol{\rho})T_K^T(\boldsymbol{\rho})
\end{equation}
where:
\begin{align*} 
\overrightarrow{\omega_K}(\boldsymbol{\rho})_\times=
\left[ \begin{array}{ccc}
0               & -\omega_z         & \omega_y \\
\omega_z        &  0                & -\omega_x  \\
-\omega_y       & \omega_x          & 0
\end{array} \right]
\end{align*}
$\omega_x, \omega_y, \omega_z$ are the coordinates of $\overrightarrow{\omega_K}(\boldsymbol{\rho})$, and
\begin{align*} 
T_K(\boldsymbol{\rho})=\prod_{1}^{K}
\left[ \begin{array}{ccc}
\cos \beta_{k} & -\sin \beta_{k} & 0\\
\sin \beta_{k} &  \cos \beta_{k} & 0\\
0                &                 0 & 1
\end{array} \right]
\left[ \begin{array}{ccc}
1                &                 0 & 0             \\
0                &     \cos \rho_{k} & -\sin \rho_{k}\\
0                &     \sin \rho_{k} &  \cos \rho_{k}
\end{array} \right]
\end{align*}

By associating each panel with the instantaneous rotation axis, we get a dual graph $C^*$ of $C$, which is formed by joining corresponding ends of the instantaneous rotation axes. $\dot{\boldsymbol{\rho}}$ is admissible if and only if $C^*$ is parallel to $C$.

\subsection{When the Paper is a Polyhedron}

If the paper is a polyhedron, many conclusions from other related topics can be filled into the framework of rigid origami. One classical topic is unfolding a given polyhedron with possible cuts along the edges, and the main problem of it is to avoid self-intersection, which is discussed in Section \ref{subsection: self-intersection}. Another topic is to consider the possible rigid folding motion of a polyhedron without cutting. Based on the well-known Cauthy's Theorem, a convex polyhedron (or a doubly-covered planar polyhedron) is first-order rigid, while a concave polyhedron might be rigid-foldable. Robert Connelly gave an example of concave rigid-foldable polyhedron \cite{connelly_counterexample_1977}. The Bellow's Conjecture, saying that the volume of a rigid-foldable polyhedron is invariant under rigid folding motion, was proved and became a typical feature of such rigid-foldable polyhedron \cite{connelly_bellows_1997}. Although studying the isometry of a polyhedron is a historical problem, the research concerning the rigid folding motion between possible isometry is relatively new and there are many open problems.

\section{Geometric Analysis of the Configuration Space}

Apart from the algebraic analysis above, this problem can also be viewed from geometry. We can regard a rigidly folded state as piecewise rotation of panels along creases, and a rigid folding motion as a piecewise continuous rotation between two rigidly folded states. The consistency and boundary constraints of the configuration space are the geometric and physical compatibility of the rigid panels. 

\subsection{Connection with Spherical Linkage Folding, Spherical Duality}

Spherical linkage folding has proved useful in modelling a single-vertex creased paper, if we put the center vertex in the center of a sufficiently small sphere. Then all the sector angles correspond to a closed series of great spherical arcs (consistency constraints) that only intersect at the endpoints of all the arcs (boundary constraints), and every folding angle is the supplement of an interior angle of this spherical polygon. Possible rigid folding motion of a single-vertex creased paper can also be descibed from triangulating this spherical linkage. We provide some basic analysis for a degree 1, 2 or 3 single creased paper in Appendix \ref{app: 1,2,3}. 

We then consider the rigid-foldability of a degree-$n$ single-vertex creased paper. In planar linkage folding, the \textit{Carpenter's rule problem} is to ask whether a simple planar polygon can be moved continuously to a position where all its vertices are in convex position, the edge lengths are preserved and there is no self-intersection along the way \cite{connelly_straightening_2003}. We can see that the rigid-foldability of a single-vertex creased paper is closely linked with the spherical version of the Carpenter's rule problem.

\begin{prop}
	Some conclusions about the folding angle space of a degree-$n$ single-vertex creased paper $\{\boldsymbol{\rho}\}_v$ from spherical linkage folding.
	
	\begin{enumerate} [(1)]
		\item If a pair of inner creases is collinear, $\sum \alpha_l \ge 2\pi$. We denote the corresponding folding angles by $\rho_i, \rho_j$, then $\{\boldsymbol{\rho}\}_{v}$ has a two dimensional subspace $\rho_i = \rho_j$. 
		\item If a pair of inner creases is collinear, and $\sum \alpha_l = 2\pi$, we denote the corresponding folding angles by $\rho_i, \rho_j$, then $\{\boldsymbol{\rho}\}_{v}$ has a two dimensional subspace $\rho_i = \rho_j$, other $\rho_k=0$. All these subspaces only intersect at $\boldsymbol{0}$.
		\item If $\sum \alpha_l \in (0, 2\pi)$, then every $\alpha_l \in (0, \pi)$, and there is no pair of collinear inner creases (antipodal points). $\{\boldsymbol{\rho}\}_{v}$ is $\{\boldsymbol{\rho}_0,-\boldsymbol{\rho}_0\}$ or the union of two disjoint 0-connected subspaces, which are symmetric to $\boldsymbol{0}$ \cite{streinu_single-vertex_2004}. Note that even if $n \ge 4$, the configuration space can be two isolated points, such as $\alpha_1=\alpha_2+\alpha_3+\alpha_4$.
		\item If $\sum \alpha_l = 2\pi$ and every $\alpha_l \in (0, \pi)$, $\{\boldsymbol{\rho}\}_{v}$ is $\{\boldsymbol{0}\}$ or the union of two 0-connected subspaces \cite{streinu_single-vertex_2004}, which are symmetric to $\boldsymbol{0}$ and only intersect at $\boldsymbol{0}$. If and only if the degree of center vertex $n \ge 4$ and the interior of crease pattern is not a cross, $\{\boldsymbol{\rho}\}_{v}$ has other non-trivial 0-connected subspaces different from (2), where every folding angle can be non-zero \cite{abel_rigid_2016}.
		\item Otherwise, from the constraints of a closed spherical linkage, $\sum \alpha_l \in (2\pi, (2n-2)\pi)$. There is no further general result on the shape and 0-connectedness of the configuration space. 
	\end{enumerate}
\end{prop}

For a general creased paper, we can model it on a sphere with the \textit{Gauss map}. Translating the startpoint of the normal vectors of all panels to the center of a unit sphere, and connecting the corresponding endpoints on this sphere if they share an inner crease. This operation forms a spherical dual of a creased paper: each panel is mapped to a point, and each inner crease (or each folding angle) is mapped to a linkage. The spherical duality would be a potential tool for analyzing a large creased paper.

\subsection{When the Interior of Crease Pattern is a Forest} \label{subsection: forest}

As stated in the introduction, the positive problem in rigid origami can now be characterized as to find the conditions on sector angles for a creased paper to be rigid-foldable. It is relatively complex because the rigid-foldability of a creased paper cannot be represented by the rigid-foldability of its single creased papers. However, if we consider the case where a creased paper has no inner panel, this relatively simple structure may generate rigid-foldability.

\begin{figure}[!tb]
	\noindent \begin{centering}
		\includegraphics[width=1\linewidth]{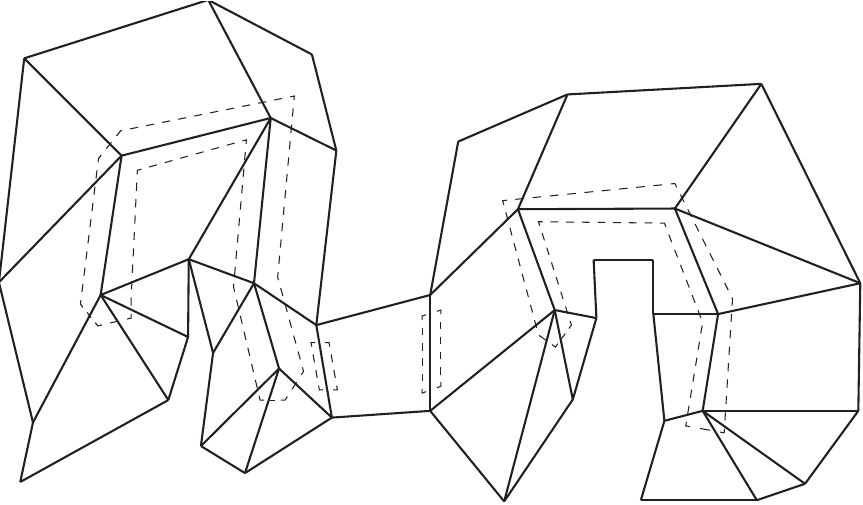}
		\par\end{centering}
	
	\caption{\label{fig: tree structure}We show a rigid-foldable creased paper with the interior of crease pattern being a forest. Here we denote each tree by dashed lines.}
\end{figure}

\begin{defn}
	In graph theory, a \textit{tree} is an undirected graph in which any two vertices are connected by exactly one path. A \textit{forest} is a disjoint union of trees.
\end{defn}

\begin{thm} \label{no cycle}
	If a creased paper $(P,C)$ satisfies:
	\begin{enumerate} [(1)]
		\item The interior of $C$ is a forest.
		\item The restriction of a rigidly folded state $(\boldsymbol{\rho},\lambda)$ on each single creased paper is rigid-foldable.    
	\end{enumerate} 
	then this rigidly folded state $(\boldsymbol{\rho},\lambda)$ is generically rigid-foldable. Especially, if $\boldsymbol{\rho}=\boldsymbol{0}$, the configuration $(\boldsymbol{\rho},\lambda)=(\boldsymbol{0},\emptyset)$ is rigid-foldable (see Figure \ref{fig: tree structure}).
\end{thm}

\begin{proof}
	Since the interior of $C$ is a forest, adjacent single creased papers only share one inner crease. Start from an arbitrary single creased paper, called $P^1$. Because $(\boldsymbol{\rho}^1,\lambda)$ is not an isolated point (here ``isolated'' means not 0-connected to any other points), there is a path between a point $(\boldsymbol{\overline{\rho}}^1,\lambda)$ and $(\boldsymbol{\rho}^1,\lambda)$. Consider an adjacent creased paper $P^2$ with a common folding angle $\rho_{c}$. We can write $\rho_{c} \in [a_1, b_1]$ on the path in $\{(\boldsymbol{\rho},\lambda)\}^1$, similarly, $\rho_{c} \in [a_2, b_2]$ on the path in $\{(\boldsymbol{\rho},\lambda)\}^2$. Let $\rho_{c} \in [a_1,b_1] \cap [a_2, b_2]$, which is not empty because $\boldsymbol{\rho}$ exists. If $[a_1,b_1] \cap [a_2, b_2]$ is a closed interval, we can re-parametrize the two paths in $\{(\boldsymbol{\rho},\lambda)\}^1$ and $\{(\boldsymbol{\rho},\lambda)\}^2$, the direct product of them is a path in $\{(\boldsymbol{\rho},\lambda)\}_{1 \cup 2}$, and now the restriction of this rigidly folded state on $P^1 \cup P^2$ is rigid-foldable. Otherwise, if $[a_1,b_1] \cap [a_2, b_2]$ is just a point, the path may not exist. The case where the restriction of $(\boldsymbol{\rho},\lambda)$ on $\{(\boldsymbol{\rho},\lambda)\}_{1 \cup 2}$ is an isolated point is what we mean by non-generic. An example of this non-generic case is shown in Figure \ref{fig: special pattern}.
	
	Now consider the special case when $\boldsymbol{\rho}=\boldsymbol{0}$. This is different to the general case because the range of all folding angles are symmetric to $0$. Because $(\boldsymbol{0}^1,\emptyset)$ is not an isolated point, there must be a path between a point $(\boldsymbol{\overline{\rho}}^1,\lambda)$ and $(\boldsymbol{0},\emptyset)$, and from the symmetry there must be a path between $(\boldsymbol{0},\emptyset)$ and $(-\boldsymbol{\overline{\rho}}^1,\lambda)$ in $\{(\boldsymbol{\rho},\lambda)\}^1$. Name its adjacent creased paper $P^2$ and the common folding angle $\rho_{c}$. $\rho_{c} \in [-a_1, a_1]$ ($a_1 \ge 0$) on the path in $\{(\boldsymbol{\rho},\lambda)\}^1$, similarly, $\rho_{c} \in [-a_2, a_2]$ ($a_2 \ge 0$) on the path in $\{(\boldsymbol{\rho},\lambda)\}^2$. Let $\rho_{c} \in [-\min(a_1,a_2),\min(a_1,a_2)]$, we can always parametrize all the folding angles on $P^1 \cup P^2$ to a continuous path by $\rho_{c}$. If $a_1$ or $a_2=0$, the direct product of the two paths in $\{(\boldsymbol{\rho},\lambda)\}^1$ and $\{(\boldsymbol{\rho},\lambda)\}^2$ is a new path in the intersection of $\{(\boldsymbol{\rho},\lambda)\}^1$ and $\{(\boldsymbol{\rho},\lambda)\}^2$.
	
	Further, we can repeat what we have done for all single creased papers. Since the number of single creased papers is at most countable, we can obtain a path between $(\boldsymbol{\rho},\lambda)$ and another point $(\boldsymbol{\overline{\rho}},\lambda)$ in the intersection of configuration spaces of all single creased papers if each time we can add a single creased paper in the generic case. If $\lambda$ is continuous on this path, $(\boldsymbol{\rho},\lambda)$ is rigid-foldable to $(\boldsymbol{\overline{\rho}},\lambda)$ along this path. Otherwise, it may be possible to choose a subset of this path to avoid self-intersection of different single creased papers and make $\lambda$ continuous (so-called generic case). Besides, for $(\boldsymbol{0},\emptyset)$, by choosing $\boldsymbol{\rho}'$ on this path sufficiently close to $\boldsymbol{\rho}$, we can avoid self-intersection of different single creased papers and ensure $\lambda$ has no definition, so $(\boldsymbol{0},\emptyset)$ is rigid-foldable to $(\boldsymbol{\rho}',\emptyset)$ and $(-\boldsymbol{\rho}',\emptyset)$.  
\end{proof}

\begin{figure}[!tb]
	\noindent \begin{centering}
		\includegraphics[width=1\linewidth]{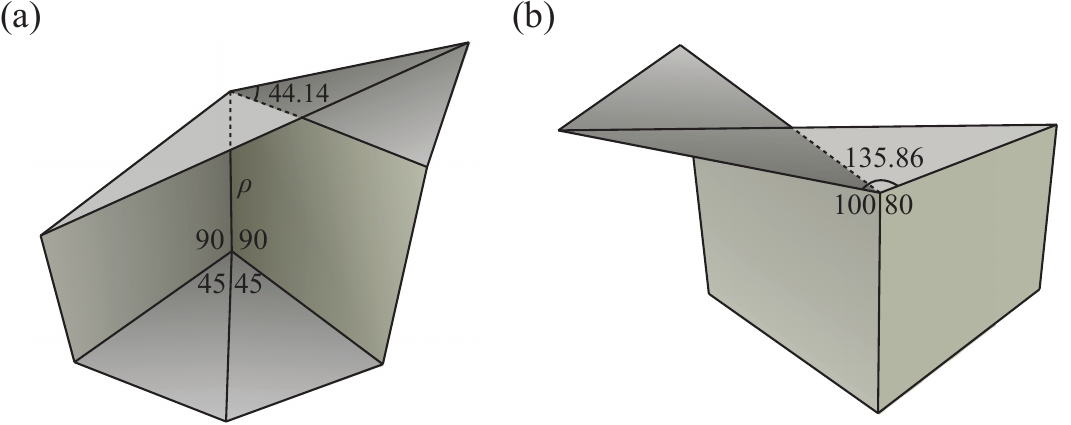}
		\par\end{centering}
	
	\caption{\label{fig: special pattern}This figure shows a non-generic case mentioned in the proof of Theorem \ref{no cycle}. (a) and (b) are the front and rear perspective views of a creased paper, where two single-vertex creased papers share two panels. The sector angles are shown in degrees. This rigidly folded state is not rigid-foldable because the common folding angle $\rho$ cannot exceed $\pi/2$ in the top single-vertex creased paper, and cannot be below $\pi/2$ in the bottom single-vertex creased paper, although both of the single-vertex creased papers are rigid-foldable.}
\end{figure} 

\begin{rem}
	Theorem \ref{no cycle} is not necessary, even when the interior of $C$ is a forest. If for some single creased papers, the restriction of $(\boldsymbol{\rho},\lambda)$ is an isolated point, it cannot have relative rigid folding motions, but $(P,C)$ can still be rigid-foldable. If there is a cycle in the interior of $C$, the path in the intersection of the configuration spaces of all single creased papers may not be successfully generated with the above one-by-one process, and generically $(\boldsymbol{\rho},\lambda)$ is not rigid-foldable. 
\end{rem}

Although the general rigid-foldability problem is hard, with Theorem \ref{no cycle}, we can analyze some simple creased papers, e.g. quadrilateral creased papers, which is discussed in a following article. 

\section{Numeric analysis of the Configuration Space}

This section gives a brief overview of the numeric analysis on rigid origami. Although this paper concentrates on the algebraic and geometric methods, numerical analysis can be an efficient tool for tracking the rigid folding motion and avoiding self-intersection. We also briefly mention the analysis of the complexity of certain problems in rigid origami to give a more complete perspective.

\subsection{Calculating the Rigid Folding Motion}

Given a creased paper $(P,C)$, we know there are at least two rigidly folded states $(\boldsymbol{t}_0, \lambda(\boldsymbol{t}_0))$ and $(-\boldsymbol{t}_0, -\lambda(\boldsymbol{t}_0))$. A possible way to know its configurations space is to solve the equation $\boldsymbol{A}(\boldsymbol{t})=\boldsymbol{0}$ introduced in Section \ref{section: algebraic} numerically, and remove the solutions in $N_{P,C}$. The solution of this system of polynomial equations is an initial value problem, where various numeric methods can be used. One example is in \cite{tachi_freeform_2010-1}.

\subsection{Self-intersection of Panels} \label{subsection: self-intersection}

In rigid origami, the study on boundary constraints (self-intersection of panels) requires different technique comparing to the study of consistency constraints. A straightforward question is, how to design a rigid folding motion without self-intersection if a given creased paper is rigid-foldable. The prime source of relevent results comes from the study of unfolding a polyhedron to a planar polygon without overlapping (also called the \textit{net}), which has been extensively studied \cite{orourke_unfolding_2008-1}. However, the result of how to design a rigid folding motion without the self-intersection when unfolding a polyhedron is still developing, as introduced below. 

It is natural to consider applying a collision detection algorithm to solve the problem of avoiding self-intersection. \cite{bohlin_path_2000} proposed an algorithm called \textit{Lazy PRM}, which minimizes the number of collision checks performed during robot motion planning and hence minimizes the running time of the planner. Then \cite{xi_continuous_2015} provided an algorithm under the Lazy PRM framework for cutting and unfolding a polyhedron continuously to one of its net without self-intersection. This algorithm also works for a ``tessellated polyhedron'', where each face of a polyhedron is triangulated densely. However, it is still a challenge to unfold a complicated non-convex polyhedron in general if the net is not \textit{linearly foldable} (``linearly foldable'' means there exists a straight-line linearly interpolating the planar state to
the fully folded state without self-intersection). Further, \cite{hao_synthesis_2018} showed that for a given polyhedron, we can optimize the way of cutting to generate an optimized net that is linearly foldable, \textit{uniformly foldable} (``uniformly foldable'' means the speed of each folding angle are the same), and much faster in motion planning than an arbitrary net. Nevertheless, for some complicated non-convex polyhedra such net may not exist.

There are also some other algorithms for some special creased papers. \cite{aloupis_edge-unfolding_2008} showed that we can unfold a ``nested band'' continuously to place all faces of the band into a plane without intersection by cutting along exactly one edge. However, the technique used here is extended from one dimensional folding, which seems hard to be applied universally. \cite{demaine_continuous_2011-1} proved that the source unfolding \cite{orourke_unfolding_2008-1} of any convex polyhedron can be continuously unfolded without self-intersection. Further, any convex polyhedron can be continuously unfolded without self-intersection after a linear number of cuts. Although here the rigid folding motion can be constructed in polynomial time, the source unfolding itself is difficult to compute.

\subsection{Complexity}

Computer-aided design is common for origami, and there have been many results on the complexity of algorithm in origami problems. Here we are more interested in analyses related to rigid origami. \cite{bern_complexity_1996} showed that it is NP-hard to determine whether a given creased paper can be fold flat (using all the inner creases). \cite{demaine_geometric_2007} showed that it only takes linear time to determine whether a single-vertex creased paper can be fold flat (using all the inner creases). A recent result \cite{akitaya_rigid_2018} showed that it is weakly NP-hard to determine whether a degree-4 creased paper (all the inner vertices are degree-4) can be folded flat (using all the inner creases), and it is strongly NP-hard to determine whether a given creased paper is rigid-foldable (using optimal inner creases). The analysis of complexity gives insight for the case where we want to approximate a target surface by sub-dividing the crease pattern.  

\section{Discussion}

\subsection{Flat-foldability of Rigid Origami}

The definition of rigid origami in this paper allows the discussion on flat-foldability. A creased paper is \textit{flat-foldable} if and only if it has a different flat rigidly folded state, where all the folding angles are $\pm \pi$. Note that flat-foldability does not require a rigid folding motion. There have been many conclusions, some of which are collected in \cite{demaine_geometric_2007}, including the flat-foldability of a strip (1-dimensional origami), a single-vertex creased paper, and the map folding. \cite{tachi_generalization_2009} provides the sufficient and necessary condition for the flat-foldability of a large quadrilateral creased paper.

\subsection{Mountain-valley Assignment}

\begin{defn}
	A \textit{mountain-valley assignment} of a creased paper is a discrete map of every inner crease $\mu\text{:}\;\{c_j\}\rightarrow\{M,V\}$. If the folding angle of an inner crease is positive, we call it a \textit{mountain crease} ($M$), while if it is negative, we call it a \textit{valley crease} ($V$).
\end{defn}

This concept is widely used. However, counting the mountain-valley assignment is known to be a difficult problem. The mountain-valley assignment is of interest here because for a developable creased paper different mountain-valley assignment can classify different branches of rigid folding motion. For a flat-foldable single-vertex creased paper, current progress is given in \cite{hull_flat_2011}. For a rigid-foldable single-vertex creased paper \cite{abel_rigid_2016} or Miura-ori \cite{ballinger_minimum_2015}, the idea of \textit{minimal forcing set} helps to analyze the mountain-valley assignment. Given a rigid-foldable creased paper and a mountain-valley assignment $\mu$, The \textit{forcing set} is a subset of inner creases such that the only possible mountain-valley assignment for this creased paper that agrees with $\mu$ on the forcing set is $\mu$ itself. If a forcing set has minimal size among all the forcing sets, it is called the \textit{minimal forcing set}. The theory for the mountain-valley assignment of a large creased paper is still developing, although given a specific example we have some techniques (at least, enumeration) to deal with it. An approach is linking this problem to graph coloring, \cite{ginepro_counting_2014} counts the number of mountain-valley assignments for local flat-foldability of a Miura-ori, while how to identify those guarantee global flat-foldability (which are the real number of branches) requires some unknown techniques. For a non-developable creased paper, the mountain-valley assignment cannot be used to classify its rigid folding motions. Generically, a rigid-foldable creased paper with more symmetry will have more branches of rigid folding motion.

\subsection{Monotonous Rigid-foldability} 

If for two rigidly folded states $(\boldsymbol{\rho}_1, \lambda_1)$ and $(\boldsymbol{\rho}_2, \lambda_2)$ of a creased paper $(P,C)$, there exists a monotonous path linking corresponding $\boldsymbol{\rho}_1$ and $\boldsymbol{\rho}_2$ in $\{\boldsymbol{\rho}\}_{P,C}$, where each component is (strictly) monotonous or remains a constant, and $\lambda$ is continuous on this path, we say $(\boldsymbol{\rho}_1, \lambda_1)$ is (\textit{strictly}) \textit{monotonously rigid-foldable} to $(\boldsymbol{\rho}_2, \lambda_2)$. 

This concept is an extension of the \textit{expansive rigid folding motion} mentioned in \cite{streinu_single-vertex_2004}. Some rigid folding motions are (strictly) monotonous, such as the folding motion of the Miura-ori. However, there are examples of rigid folding motions that are not monotonous -- an example being the rigid folding motion between some rigidly folded states of a degree-5 single-vertex creased paper. Monotonous rigid-foldability is a stronger property than 0-connectedness in the configuration space. A better understanding of monotonous rigid-foldability might prove to be useful, for instance in selecting an expansive rigid folding motion that can avoid local self-intersection.

\subsection{Variation for a Given Crease Pattern}
If we fix the interior of a crease pattern, the shape of panels and the outer creases are not related to the consistency constraints of a rigidly folded state, thus we can reshape the panels and the outer creases to obtain a different creased paper, which will not essentially change the configuration space.

\subsection{Kirigami}
Kirigami can be defined as cutting along at most countable continuous curves on a creased paper. If a cut curve is closed, a region whose boundary is this cut curve will be removed, otherwise a cut curve will be split into two boundary components. How kirigami will affect the rigid-foldability has not been fully studied, although clearly kirigami will not decrease the rigid-foldability of a creased paper. We intend to discuss this in a future article.

\section{Conclusion}
This article puts forward a theoretical framework for rigid origami, which is used to review some important previous results and to draw some new conclusions. After clarifying the related definitions, the key problem is to describe the configuration space of a creased paper, and we provide some useful results from other related areas, including rigidity theory, graph theory, linkage folding, and computer science. Although some progress has been made, the complete theory is still unclear, which will lead to future work.

\section{Acknowledgment}

We thank Allan McRobie, Tomohiro Tachi, Robert Connelly, Thomas C. Hull, Walter Whiteley and Hellmuth Stachel for helpful discussions.


\section*{Appendix}

\appendix

\subsection{An Alternative Definition of Rigid Origami from Origami} \label{app: alternative}

Here we present the result in \cite{demaine_non_2011}: an isometric map on a creased paper will become piecewise rigid if we require the paper $S$, crease pattern $G$ and isometry function $f$ to have stronger properties. Following Definitions \ref{defn: generalized}--\ref{defn: folded state}, we add the conditions below,
\begin{enumerate} [(1)]
	\item $S$ is piecewise-$C^2$.
	\item If a point on a crease $c$ is $C^0$ or $c \subset \partial S$, $c$ is a line segment.
	\item A point on a crease or piece of $(S,G)$ is locally isometric to a disk or a half-disk. A vertex of $(S,G)$ is not necessarily locally isometric to a disk or a half-disk.
	\item $f$ is an isometry function such that $f(S)$ is piecewise-$C^2$.
\end{enumerate}
Then if all ``folding angles'' are non-trivial,
\begin{enumerate} [(1)]
	\item Each piece is planar.
	\item The restriction of $f$ on each piece is a combination of translation and rotation (reflection is not necessary).
\end{enumerate}

\subsection{The Configuration Space of a Degree-1, 2 and 3 Single Creased Paper} \label{app: 1,2,3}

Here We consider the folding angle space of a degree-$n$ single-vertex and single-hole creased paper, called $\{\boldsymbol{\rho}\}_v^n$ and $\{\boldsymbol{\rho}\}_h^n$, from the solution space $W_v^n$ and $W_h^n$. When $n \le 3$, the order function has no definition.

\begin{enumerate} [(1)]
	\item $W_v^1$: $\rho_1=0$, when $\alpha_1=2\pi$. $\{\boldsymbol{\rho}\}_v^1=W_v^1$, which is the same for $W_h^1$ and $\{\boldsymbol{\rho}\}_h^1$ with $\beta_1=2\pi$ and $a_1=b_1=0$. 
	
	The folding angle on an inner crease incident to a degree-1 vertex is always $0$, which means we can regard this single creased paper as a panel. On the other hand, if a single-hole creased paper has only one inner crease, this single creased paper as well as the hole should always keep planar, which means we can merge them to a panel. Therefore in a large creased paper we only need to consider at least degree-2 single creased papers.
	
	\item $W_h^2$ is either (for $W_v^2$, set $\alpha_1=\beta_1$, $\alpha_2=\beta_2$, $a_1=b_1=a_2=b_2=0$)
	\begin{enumerate}
		\item $\rho_1=\rho_2$, when $\beta_1 = \beta_2=\pi$, $a_1=a_2$, $b_1=b_2=0$.
		\item $\rho_1=\rho_2=0$, when $\beta_1 + \beta_2=2\pi$, and
		\begin{equation} \label{degree-2 single hole 1}
		\begin{split}
		a_1+a_2 \cos \beta_1-b_2 \sin \beta_1=0 \\
		b_1+a_2 \sin \beta_1+b_2 \cos \beta_1=0
		\end{split}
		\end{equation}		
		\item $\rho_1= \pm \pi$, $\rho_2=\pm \pi$, when $\beta_1 = \beta_2$, and
		\begin{equation} \label{degree-2 single hole 2}
		\begin{split}
		a_1+a_2 \cos \beta_1+b_2 \sin \beta_1=0 \\
		b_1+a_2 \sin \beta_1-b_2 \cos \beta_1=0
		\end{split}
		\end{equation}
	\end{enumerate}
	$\{\boldsymbol{\rho}\}_h^2$ is either (for $\{\boldsymbol{\rho}\}_v^2$, set $\alpha_1=\beta_1$, $\alpha_2=\beta_2$, $a_1=b_1=a_2=b_2=0$)
	\begin{enumerate}
		\item $\rho_1=\rho_2$, when $\beta_1 = \beta_2=\pi$, $a_1=a_2$, $b_1=b_2=0$. 
		\item $\rho_1 = \rho_2 = 0$, when $\beta_1 + \beta_2= 2\pi$, and equation (\ref{degree-2 single hole 1}) is satisfied.
		\item $\rho_1 = \rho_2 = \pm \pi$, when $\beta_1=\beta_2$, and equation (\ref{degree-2 single hole 2}) is satisfied.
	\end{enumerate}
	
	Considering $\{\boldsymbol{\rho}\}_v^2$, for a degree-2 single-vertex creased paper in case (a), we can merge the vertex and two inner creases into one inner crease; for case (b) or (c), we can regard this single creased paper as a panel. For a degree-2 single-hole creased paper, case (a) can be regarded as two panels rotating along an inner crease. For case (b) or (c), the configuration space is trivial, which means we can regard them as a panel. Therefore in a large creased paper we only need to consider at least degree-3 single creased papers. 
\end{enumerate}

For $n \ge 3$, it seems hard to make direct symbolic calculations and study the real roots. A possible way is to find some properties of the solution space from analyzing the Reduced Gr\"obner Basis, but the complexity increases rapidly. We then provide the result for $\{\boldsymbol{\rho}\}_{v}^{3}$ from the analysis on a spherical triangle, which is:

\begin{enumerate} [(a)]
	\item $i, j, k$ is a permutation of $\{1,2,3\}$. If 
	\begin{enumerate}
		\item $\alpha_i=\pi$, then $\alpha_j+\alpha_k=\pi$, $\rho_k=\rho_i$, $\rho_j=0$.
		\item $\alpha_i+\alpha_j=\pi$, then $\alpha_k=\pi$, $\rho_k=\rho_j$, $\rho_i=0$.
	\end{enumerate}
	\item If $\alpha_1+\alpha_2+\alpha_3=2\pi$ and not (a), $\rho_1=\rho_2=\rho_3=0$. 
	\item Otherwise, there are two solutions $\{\rho_1, \rho_2, \rho_3\}$ and $\{-\rho_1, -\rho_2, -\rho_3\}$, which satisfies the following equations and the supplement of $\rho_1, \rho_2, \rho_3$ are the interior angles of a spherical triangle. (Special cases like $\alpha_i=\alpha_j+\alpha_k$ are included.)   
	\begin{equation}
	\begin{split}
	\cos \rho_1=\dfrac{\cos \alpha_1 \cos \alpha_2-\cos \alpha_3}{\sin \alpha_1 \sin \alpha_2} \\
	\cos \rho_2=\dfrac{\cos \alpha_2 \cos \alpha_3-\cos \alpha_1}{\sin \alpha_2 \sin \alpha_3} \\
	\cos \rho_3=\dfrac{\cos \alpha_3 \cos \alpha_1-\cos \alpha_2}{\sin \alpha_3 \sin \alpha_1}
	\end{split}
	\end{equation}
\end{enumerate}

As for a degree-3 single-hole creased paper, $\{\boldsymbol{\rho}\}_{h}^{3}$ is a subset of $\{\boldsymbol{\rho}\}_{v}^{3}$ and is not empty. Since it is not possible for a degree-3 single creased paper to have continous rigid folding motion different from folding along a single crease, we usually consider no less than degree-4 single creased papers.

\vskip2pt

\bibliographystyle{RS}

\bibliography{Rigid-Folding}

\end{document}